\documentclass[reqno,11pt]{amsart}
\usepackage{amsmath,amsthm}
\usepackage{latexsym}
\usepackage{amssymb}
\usepackage[all,cmtip]{xy}

\usepackage{xcolor}

\newtheorem{Proposition}{Proposition}[section]
\newtheorem{Definition}[Proposition]{Definition}
\newtheorem{Lemma}[Proposition]{Lemma}
\newtheorem{Theorem}[Proposition]{Theorem}
\newtheorem{MainTheorem}{Theorem}
\newtheorem{Corollary}[Proposition]{Corollary}

\def\R{\mathbb{R}}
\def\Z{\mathbb{Z}}

\def\C{\mathbb{C}}

\def\O{\mathbb{O}}

\def\S{\mathbb{S}}

\def\k{\mathcal{K}}

\def\V{\mathcal{V}}
\def\P{\mathcal{P}}
\def\W{\mathcal{W}}
\def\so{\mathfrak{so}}
\def\SO{\mathrm{SO}}

\def\Kl{\mathrm{Kl}}
\def\Val{\mathrm{Val}}
\def\stab{\mathrm{Stab}}
\def\spin{\mathrm{Spin}}
\def\Vst{V_{st}}
\DeclareMathOperator{\vol}{vol}
\DeclareMathOperator{\nc}{nc}
\DeclareMathOperator{\Hom}{Hom}
\DeclareMathOperator{\Char}{Char}
\DeclareMathOperator{\Gr}{Gr}

\newcommand{\riem}{\mu^\text{sec}}

\title{$\spin(9)$-invariant valuations on the octonionic plane}
\author{Andreas Bernig}
\author{Floriane Voide}
\thanks{Supported by SNF-grant PP002-114715/1 and DFG-grant BE 2484/5-1}
\address{Institut f\"ur Mathematik, Johann Wolfgang Goethe-Universit\"at Frankfurt\\
Robert-Mayer-Str. 10, 60054 Frankfurt, Germany}
\email{bernig@math.uni-frankfurt.de}
\email{voide@math.uni-frankfurt.de} 

\begin{document}

\begin{abstract}
The dimensions of the spaces of $k$-homogeneous $\spin(9)$-invariant valuations on the octonionic plane are computed using results from the theory of differential forms on contact manifolds as well as octonionic geometry and representation theory. Moreover, a valuation on
Riemannian manifolds of particular interest is constructed which yields, as a special
case, an element of $\Val_2^\spin(9)$.
\end{abstract}

\maketitle

\section{Introduction}

On an $n$-dimensional Euclidean vector space $V$, we consider the space $\k(V)$ of compact convex subsets in $V$. A {\bf valuation} is a real or complex valued functional $\mu$ on $\k(V)$ with the property that
\begin{displaymath}
\mu(K\cup L)+\mu(K\cap L)=\mu(K)+\mu(L), 
\end{displaymath}
for all $K,L\in\k(V)$ with $K\cup L\in\k(V)$.

In this article we consider only continuous (with respect to the Hausdorff topology on $\k(V)$) and translation invariant valuations; we denote by $\Val$ the vector space of these valuations.

A valuation $\mu$ is said to be {\bf $k$-homogeneous} if 
\begin{displaymath}
\mu(tK)=t^k\mu(K) 
\end{displaymath}
for every $K\in\k(V)$ and $t>0$. We denote by $\Val_k$ the subspace of $k$-homogeneous valuations in $\Val$. By a theorem of McMullen \cite{mcmullen77} we have the following decomposition
\begin{equation} \label{eq_mcmullen}
 \Val=\bigoplus_{k=0}^n\Val_k.
\end{equation}

A valuation $\mu$ is \textbf{even} if $\mu(-K)=\mu(K)$ and \textbf{odd} if $\mu(-K)=-\mu(K)$ for all $K\in\k(V)$.

Let $G$ be a compact subgroup of the special linear group $\SO(n)$ of $V$, and let $\bar G$ be the group generated by $G$ and translations. Consider the space $\Val^G\subset \Val$ of $\bar G$-invariant continuous valuations. It was shown by Alesker \cite{alesker00} that this vector space has finite dimension if and only if $G$ acts transitively on the 
unit sphere $S(V)$ in $V$. A classification of all compact connected Lie groups acting transitively and effectively on $S(V)$ was obtained by Montgomery-Samelson \cite{montgomery_samelson43} and Borel \cite{borel49}.
The list contains the series 
\begin{equation}\label{eq:series}
\SO(n),\mathrm{U}(n),\mathrm{SU}(n), \mathrm{Sp}(n),\mathrm{Sp}(n)\cdot\mathrm{U}(1),\mathrm{Sp}(n)\cdot \mathrm{Sp}(1),
\end{equation}
and the three exceptional groups
\begin{equation}\label{eq:excep}
\mathrm{G}_2,\spin(7),\spin(9).
\end{equation}

The space of continuous $\overline{\SO}(n)$-invariant valuations is described by a famous theorem of Hadwiger \cite{hadwiger_vorlesung}, which states that $\Val^{\SO(n)}$ has dimension $n+1$. A basis of this space is given by the so called intrinsic volumes. Hadwiger's theorem implies in a very straightforward way many of the most important theorems in integral geometry, compare \cite{klain_rota} for these applications.  

Hadwiger-type theorems and integral geometry of the unitary groups $\mathrm{U}(n),\mathrm{SU}(n)$ was extensively studied during the last few years \cite{alesker_mcullenconj01, bernig_sun09, bernig_fu_hig, bernig_fu_solanes, park02, wannerer_area_measures, wannerer_unitary_module}.

For the symplectic groups $\mathrm{Sp}(n),\mathrm{Sp}(n)\cdot\mathrm{U}(1),\mathrm{Sp}(n)\cdot \mathrm{Sp}(1)$, the dimensions of the spaces of invariant valuations were computed in \cite{bernig_qig}. A Hadwiger-type theorem is known only in the cases $n=1$ \cite{alesker_su2_04, bernig_sun09, bernig_quat09} (note that $\mathrm{Sp}(1) \cong \mathrm{SU}(2)$) and $n=2$ \cite{bernig_solanes}. 

The cases of $\mathrm G_2$ and $\spin(7)$ have been treated by the first named author in \cite{bernig_g2}.

Until now, only partial results have been obtained by Alesker in the case $G=\spin(9)$, acting on a $16$-dimensional space by the spin representation. Using the theory of plurisubharmonic functions in quaternionic variables, he constructed a 2-homogeneous $\spin(9)$-invariant valuation in \cite{alesker08}.

In the first section of this paper we compute the dimensions of the spaces $\Val_k^{\spin(9)}$.
\begin{MainTheorem} \label{mainthm_dimensions}
\begin{displaymath}
  \begin{array}{| c || c | c | c | c | c | c | c | c | c | c | c | c | c | c | c | c | c |} \hline
   k & 0 & 1 & 2 & 3 & 4 & 5 & 6 & 7 & 8 & 9 & 10 & 11 & 12 & 13 & 14 & 15 & 16\\ \hline
   \dim\Val_k^{\spin(9)} & 1 & 1 & 2 & 3 & 6 & 10 & 15 &20 & 27 & 20 & 15 & 10 & 6 & 3 & 2 & 1 & 1 \\ \hline
  \end{array}
\end{displaymath}
\end{MainTheorem}

These dimensions are computed using the results of the first named author \cite[Theorem 2.1 and Corollary 2.2]{bernig_qig}, the description of the spin action due to Sudbery \cite{sudbery84} and facts from representation theory.

Before going on, let us compare our result with the one in \cite{bernig_g2}. The two exceptional groups $\mathrm{G}_2$ and $\spin(7)$ act not only transitively on some sphere, but isotropically, i.e. the stabilizer acts transitively on the unit sphere of the tangent space. This allows to relate invariant valuations under these groups to invariant valuations on the tangent spaces. This approach does not work for $\spin(9)$, since the action of this group on $S^{15}$ is not isotropic. Another difference is that (again in contrast to the cases $\mathrm{G}_2$ and $\spin(7)$) there are rather many new invariant valuations. Therefore it seems a difficult problem to write down an explicit Hadwiger-type theorem and to compute the kinematic formulas for this group. 

In the second part of this paper, we present a result in the more general setting of valuations on manifold. Before stating it, we have to recall the Klain embedding. 

Let $\mu$ be an even valuation of homogeneity degree $k$ on the Euclidean vector space $V$. For a $k$-plane $E\in \Gr_k(V)$, the restriction $\mu|_E$ is a continuous, translation invariant simple valuation. From a characterization theorem of Klain \cite{klain95}, it follows that $\mu|_E = c(E) \cdot \vol_E$ for some constant $c(E) \in \R$. 
\begin{Definition}
The map $\Kl_\mu:\Gr_k(V)\longrightarrow\R, \Kl_\mu(E)=c(E)$ is called the 
{\bf Klain function} of $\mu$. It is a continuous function on $\Gr_k(V)$.
\end{Definition} 

The induced map $\Kl:\Val_k^+\to C(\Gr_k(V)),\mu\mapsto \Kl_\mu$ is injective. $\Kl$ is called the {\bf Klain embedding}.

Let us now recall some notions related to Alesker's theory of valuations on manifolds \cite{alesker_val_man1, alesker_val_man2, alesker_survey07, alesker_val_man4, alesker_val_man3}. Roughly speaking, a smooth valuation on a manifold $M$ of dimension $n$ is a functional (satisfying some technical properties) on the space of smooth compact submanifolds with corners (see Section \ref{sec_valman} for the definition). The space $\mathcal{V}^\infty(M)$ of smooth valuations on $M$ admits a natural filtration  
\begin{displaymath}
 \V^\infty(M)=\W_0 \supset \W_1 \supset \cdots \supset \W_n.
\end{displaymath}
Under the action of the Euler-Verdier involution (see Section \ref{sec_valman}), each $W_k$ splits as a sum $W_k \cong W_k^+ \oplus W_k^-$. The quotient $\W_k/\W_{k+1}$ is canonically isomorphic to the space of smooth sections of the vector bundle $\Val_k(TM)$ over $M$ whose fiber at a point $p \in M$ is given by $\Val_k(T_pM)$. If $\mu \in \W_k$, we denote by $T_p^k \mu$ the image of $\mu$ in $\Val_k(T_pM)$. If $\mu \in W_k^\epsilon$ with $\epsilon:=(-1)^k$, then $T_p^k \mu \in \Val_k^+(T_pM)$ is even and can be described by its Klain function, which is a function on $\Gr_k(T_pM)$. 

Let $(M,g)$ be a Riemannian manifold, $p \in M$. The sectional curvature of $M$ at $p$ is a function $K_p \in C(\Gr_2 T_pM)$. Note that the Klain function of an even $2$-homogeneous valuation on $T_pM$ belongs to the same space. 

\begin{MainTheorem} \label{mainthm_sectional}
 On every Riemannian manifold $(M,g)$, there exists a valuation $\mu \in \W_2^+$ such that for every $p\in M$, the Klain function of the induced valuation $T_p^2\mu$ on $T_pM$ is the sectional curvature of $M$ at $p$.
\end{MainTheorem}

The example $M=\O P^2$ (octonionic projective plane) yields an even, homogeneous of degree $2$, $\spin(9)$-invariant valuation on $\O^2=\R^{16}$, which is not a multiple of the second intrinsic volume. 

\textbf{Acknowledgements.} We are very grateful to Franz Schuster and Gil Solanes for their suggestions and comments on the first drafts of this paper.

\section{Forms and valuations} 

We begin by recalling the ideas and notations from \cite{bernig_qig}, leading to Theorem \ref{thm_exact_seq} and Corollary \ref{cor_dimension_formula} below.

Let $V$ be an $n$-dimensional Euclidean vector space. The unit sphere in $V$ will be denoted by $S(V)$. The sphere bundle $SV:=V\times S(V)$ is a contact manifold with contact form given by
\begin{equation*}
 \alpha_{(x,v)}(w)=\langle v,d\pi(w)\rangle, \quad w \in T_{(x,v)}SV,
\end{equation*}
where $\pi:SV\rightarrow V$ is the canonical projection. We denote by $Q:=\mathrm{ker} \alpha$ the contact distribution.

The space of complex valued differential forms on $SV$ is denoted by $\Omega^*(SV)$. It has a bigrading
\begin{displaymath}
 \Omega^*(SV)=\bigoplus_{k,l}\Omega^{k,l}(SV),
\end{displaymath}
where $\Omega^{k,l}(SV)$ denotes the space of differential forms of bidegree $(k,l)$ on $SV$. 

\begin{Definition}\label{def:nc}
Given $K \in \k(V)$, the oriented $(n-1)$-dimensional Lipschitz submanifold of $SV$
\begin{displaymath}
 N(K):=\{(x,v)\in SV\,|\,v\ \mathrm{is\ an\ outer\ unit\ normal\ vector\ of\ }K\ \mathrm{at}\ x\}
\end{displaymath}
is called the {\bf normal cycle} of $K$.
\end{Definition}

For basic properties of $N(K)$ we refer to \cite{fu94}. 

\begin{Definition}
A translation invariant valuation $\mu \in \Val$ is called \textbf{smooth} if there exist $\omega \in \Omega^{n-1}(SV)^{tr}$ and $\varphi \in \Lambda^n V^* \otimes \C=\Omega^n(V)^{tr}$ such that
\begin{displaymath}
\mu(K)=\int_{N(K)}\omega+\int_K \varphi, 
\end{displaymath}
for every $K \in \k(V)$, where the superscript $tr$ means translation invariance.
We denote by $\Val^{\infty}$ the space of translation invariant smooth valuations.
\end{Definition}

Smooth valuations are dense in the space of continuous translation invariant valuations. In recent years, several algebraic operations like product, convolution and Alesker-Fourier transform were introduced on $\Val^\infty$ \cite{alesker04_product, alesker_fourier, bernig_aig10, bernig_fu06, fu_barcelona}. Moreover, these algebraic structures are closely related to kinematic formulas, compare \cite{bernig_fu_hig, bernig_fu_solanes, bernig_hug, wannerer_area_measures}.  

The space $\Val^\infty$ fits into an exact sequence, as was shown in \cite{rumin94} and \cite{bernig_qig} and will be explained now. We consider the following subspaces of $\Omega^{k,l}(SV)^{tr}$.

\begin{align*}
 \mathcal{I}^{k,l}(SV)^{tr} & := \{\omega\in\Omega^{k,l}(SV)^{tr}\, |\, \omega=\alpha\wedge\xi+d\alpha\wedge\psi,\, \xi\in\Omega^{k-1,l}(SV)^{tr}\\
 & \quad \psi \in \Omega^{k-1,l-1}(SV)^{tr}\}, \\
\Omega^{k,l}_v(SV)^{tr} & := \{\omega\in\Omega^{k,l}(SV)^{tr}\, |\, \alpha \wedge \omega=0\}, \\
\Omega^{k,l}_h(SV)^{tr} & := \left. \Omega^{k,l}(SV)^{tr} \middle/ \Omega^{k,l}_v(SV)^{tr} \right.,\\
\Omega^{k,l}_p(SV)^{tr} & := \left. \Omega^{k,l}(SV)^{tr} \middle/ \mathcal{I}^{k,l}(SV)^{tr}\right..
\end{align*}
The letters v,h,p stand for \textbf{vertical, horizontal, primitive} respectively. 

Let $L$ be multiplication by the $2$-form $d\alpha$. Then $L:\Omega_h^{k-1,l-1}(SV)^{tr} \to \Omega_h^{k,l}(SV)^{tr}$ is an injection for $k+l \leq n$ and 
\begin{equation} \label{eq_primitive_horizontal}
 \Omega_p^{k,l}(SV)^{tr} \cong \Omega_h^{k,l}(SV)^{tr}/L \Omega_h^{k-1,l-1}(SV)^{tr}.
\end{equation}

Let $G$ be a compact subgroup of $\SO(V)$ acting transitively on $S(V)$. By $\bar G$ we denote the group generated by $G$ and translations. By a slight abuse of notation, the superscript $G$ stands for translation and $G$-invariance. 

We consider the operators
\begin{displaymath}
 d_Q:\Omega^{k,l}_p(SV)^G \to \Omega^{k,l+1}_p(SV)^G
\end{displaymath}
induced by the exterior derivative $d:\Omega^{k,l}(SV)^G \to \Omega^{k,l+1}(SV)^G$, and
 
\begin{align*}
 \nc:\Omega_p^{k,n-k-1}(SV)^G & \rightarrow \Val_k^G\\
\omega & \mapsto \int_{N(\cdot)}\omega.
\end{align*}

This map is well defined, since $N(K)$ is a legendrian cycle, hence vanishes on $\mathcal{I}^{k,n-k-1}(SV)^G$.

\begin{Theorem}[\cite{bernig_qig}]\label{thm_exact_seq}
For $0<k<n$, the sequence 
\begin{multline*}
\xymatrix{
0\ar[r] &(\Lambda^kV^*)^G \otimes \C \ar[r] & \Omega_p^{k,0}(SV)^G\ar[r]^-{d_Q} & \cdots} \\
\xymatrix{\qquad\qquad\qquad\qquad\qquad\qquad \cdots \ar[r]^-{d_Q} &\Omega_p^{k,n-k-1}(SV)^G \ar[r]^-{\nc} & 
\Val_k^G \ar[r] &0
} 
\end{multline*}
is exact.
\end{Theorem}

Note that exactness on the left part of the sequence was shown by Rumin \cite{rumin94}, while exactness of the right hand side is a consequence of results in \cite{bernig_broecker07}.

The theorem and \eqref{eq_primitive_horizontal} yield the following corollary. 

\begin{Corollary}[\cite{bernig_qig}]\label{cor_dimension_formula}
 For $0\leq k,l\leq n$, set
\begin{align*}
 b_k & :=\dim(\Lambda^k V)^G,\\
 b_{k,l} & :=\dim \Omega^{k,l}_h(SV)^G,
\end{align*}
and $b_k:=0$, $b_{k,l}:=0$ for other values of $k$ and $l$. Then for $0\leq k\leq n$ :
\begin{equation}\label{eq1}
 \dim \Val_k^G=\sum_{l=0}^{n-k-1}(-1)^{n-k-l-1}(b_{k,l}-b_{k-1,l-1})+(-1)^{n-k}b_k.
\end{equation}
\end{Corollary}

The first named author \cite{bernig_qig} used this corollary to determine the dimensions of the spaces of invariant valuations on quaternionic vector spaces. In the present paper, we will study the case of the octonionic plane $\O^2$.

\section{The group $\spin(9)$} \label{sec_spin9}

Let us first recall the definition of the spin groups $\spin(n)$ for general $n$.

On a vector space $V$ with dimension $n$ we consider the special orthogonal group $\SO(V) \cong \SO(n)$ of $V$.

It is well known that the fundamental group of $ \SO(n)$ is given by $\pi_1(\SO(n))=\mathbb{Z}/2\mathbb{Z}$ for $n\geq 3$. This implies that, for $n\geq3,$ $\SO(n)$ has a connected double covering, called the {\bf spin group} $\spin(n)$.

Explicit constructions and descriptions of these groups can be found in \cite{broecker_tomdieck, fulton_harris91}. Note that the Lie algebra of $\spin(n)$ equals the Lie algebra of $\SO(n)$, that is the space $\so(n)$ of anti-symmetric matrices.    

The group $\spin(9)$ admits a representation on a $16$-dimensional space, called \textbf{spin representation}. We will use the following description of the spin representation on the Lie algebra level due to Sudbery (\cite{sudbery84}). Recall that $\O$ denotes the 8-dimensional normed division algebra of the octonions. $\O$ is neither commutative nor associative. However, it is alternative, i.e. for all $a,b,c\in\O$,
\begin{displaymath}
[a,b,c]=-[b,a,c], 
\end{displaymath}
where $[a,b,c]:=a(bc)-(ab)c$ is the \textbf{associator}. The center of $\O$ is $\R$ and we denote by $\O'$ its orthogonal complement in $\O$, i.e. the $7$-dimensional space  of pure octonions.

\begin{Proposition}[Triality principle, \cite{sudbery83}]
 For $T\in\so(\O)$, there exist unique elements $T^\sharp,T^\flat\in\so(\O)$ satisfying the following generalization of the derivation equation:
\begin{displaymath}
 T(xy)=(T^\sharp x)y+x(T^\flat y).
\end{displaymath}
\end{Proposition}

\begin{Theorem}[\cite{sudbery84}]\label{action}
\begin{enumerate}
 \item The Lie algebra $\so(9)$ of $\spin(9)$ can be represented as
\begin{equation*}
 \so(9)=A_2'(\O)\oplus \so(\O'),
\end{equation*}
where $A_2'(\O)$ is the set of traceless antihermitian $2\times 2$ matrices with entries in $\O$.
\item The spin representation $\rho$ of $\so(9)$ on 
\begin{displaymath}
 S:=\{2\times1\ \mathrm{column\ vectors\ with\ entries\ in\ } \O\} \cong \O^2\cong\R^{16}
\end{displaymath}
is given by
\begin{align*}
\rho(A)(x) & :=A\cdot x,\quad A\in A_2'(\O) \quad \text{(matrix multiplication)},\\
\rho(T)(x) & :=T^\sharp x,\quad T\in \so(\O') \quad \text{(componentwise action)},
\end{align*}
where $T^\sharp\in\so(\O') \subset \so(\O)$ is uniquely defined by the triality principle.
\end{enumerate}
\end{Theorem}

\begin{Corollary}
The stabilizer of the $\spin(9)$-action on $S^{15}$ is $\spin(7)$. Its action on the tangent space $T_{(1,0)}\S^{15}=\O'\oplus\O$ is the sum of the standard representation on $\O'=\R^7$ and the spin representation on $\O=\R^8$. 
\end{Corollary}

For the details of these computations, we refer to \cite{voide_thesis}.

Let us also recall some facts from the representation theory for odd dimensional orthogonal Lie algebras, again referring to \cite{broecker_tomdieck, fulton_harris91} for details.

An irreducible representation of $\so(2m+1)$ can be represented by an element of the lattice $\Lambda\subset \R^m$ generated by $L_1,\ldots,L_m$ and $(L_1+\cdots+L_m)/2$, i.e. by an element of the form
\begin{displaymath}
 \sum_{i=1}^m\lambda_iL_i
\end{displaymath}
with $\lambda_1 \geq \cdots \geq\lambda_m\geq0$ and the $\lambda_i$'s are either all integers or all half integers. We encode this information in the vector $[\lambda_1,\ldots,\lambda_m]$, called {\bf highest weight}, and the associated irreducible representation is denoted by $\Gamma_{[\lambda_1,\ldots,\lambda_m]}$.

\begin{Proposition}\cite{fulton_harris91}
For $k=1,\ldots,m-1$, the exterior power $\Lambda^k(V_{st})$ of the standard representation $V_{st}$ of $\so(2m+1)$ is the irreducible representation with highest weight \begin{displaymath}
[\underbrace{1,\ldots,1}_{k\ \mathrm{ terms}},0,\ldots,0].                                                                                                                                                                    
\end{displaymath}

The spin representation $S$ is the irreducible representation with highest weight
\begin{displaymath}
[1/2,\ldots,1/2]. 
\end{displaymath}
\end{Proposition}

These $m$ representations are called fundamental representations. In fact, the representation ring $R$ is a polynomial ring on the isomorphism classes of the fundamental representations 
\begin{displaymath}
 R=\mathbb{Z[}[V_{st}],[\Lambda^2V_{st}],\ldots,[\Lambda^{m-1}V_{st}],[S]].
\end{displaymath}

It is useful to work with characters of representations, since the character carries the essential information about the representation. Let us denote by $\Z[\Lambda]$ the integral group ring on the abelian group $\Lambda$.

The character of the irreducible representation $\Gamma_{[\lambda_1,\ldots,\lambda_m]}=\Gamma_\lambda$ can be computed by Weyl's character formula (\cite{fulton_harris91})
\begin{equation} \label{eq_weyl}
\Char(\Gamma_\lambda)=\frac{\left|x_j^{\lambda_i+m-i+1/2}-x_j^{-(\lambda_i+m-i+1/2)}\right|}{\left|x_j^{m-i+1/2}-x_j^{-(m-i+1/2)}\right|}\in \Z[\Lambda], 
\end{equation}
where $x_j^{\pm1}$ and $x_j^{\pm 1/2}$ are the elements of $\Z[\Lambda]$ corresponding to the weights $\pm L_j$ and $\pm\frac{1}{2}L_j$ respectively.

We denote by $B_i$ the character of $\Lambda^iV_{st}$ and by $B$ the character of the spin representation $S$. Any $\so(2m+1)$-representation $V\in R$ can be expressed as a polynomial in the fundamental representations $V_{st},\Lambda^2V_{st},\ldots,\Lambda^{m-1}V_{st},S$, and its character can be computed as the same polynomial in $B_1,\ldots,B_{m-1},B$ (\cite{fulton_harris91}).

In practice, there is no easy way to find the decomposition of an arbitrary representation as an element of 
$R=\Z[[\Vst],[\Lambda^2\Vst],\ldots,[\Lambda^m\Vst],[S]]$. But we will see that the only representations 
of interest in our case are exterior powers of irreducible representations or sums of irreducible representations. To compute the 
character of a representation given in this form, we can use the following recurrence formula.

\begin{Theorem}[Adams formula \cite{fulton_harris91}]\label{theo:adams}
Define the {\bf Adams operator} $\psi^k:\Z[\Lambda]\longrightarrow \Z[\Lambda]$ by  $\psi^k(x_j)=x_j^k$. Then for any $\so(2m+1)-$representation $V$
\begin{displaymath}
\Char(\Lambda^d V) = \frac{1}{d}\sum_{k=1}^d (-1)^{k-1}\psi^k(\Char V) \Char(\Lambda^{d-k}V). 
\end{displaymath}
\end{Theorem}

This formula allows us to compute inductively the character of $\Lambda^kV$. The next step is to write the obtained polynomial as a 
linear combination of characters of irreducible representations. Since a character uniquely determines the associated representation, we conclude that, if
\begin{displaymath}
 \Char(V)=\bigoplus n_\lambda \Char(\Gamma_\lambda),
\end{displaymath}
then
\begin{displaymath}
 V=\bigoplus n_\lambda \Gamma_\lambda.
\end{displaymath}

To decompose the character of a representation in characters of irreducible representations, we need two observations 
\begin{itemize}
 \item The leading monomial of the character of the irreducible representation $\Gamma_{[\lambda_1,\ldots,\lambda_m]}$ is
$\ x_1^{\lambda_1}\cdot \ldots \cdot x_m^{\lambda_m}.$ Recall that the leading monomial of a polynomial is the monomial of highest degree (with respect to the lexicographic order).
 \item If the leading monomial of the character of a representation $V$ is $n_\lambda x_1^{\lambda_1}\cdot \ldots \cdot x_m^{\lambda_m}$, 
then the leading monomial of the character of $V-n_\lambda \Gamma_{[\lambda_1,\ldots,\lambda_m]}$ is of strictly lower degree. 
\end{itemize}

Therefore we apply the following algorithm to decompose the character of a representation $V$ in irreducible characters :
\begin{enumerate}
 \item[a)] Find the leading monomial of $\Char(V)$ : $ n_\lambda x_1^{\lambda_1}\cdot \ldots \cdot x_m^{\lambda_m}$.
 \item[b)] Compute $\Char(\Gamma_{[\lambda_1,\ldots,\lambda_m]})$ with the help of Weyl's character formula.
 \item[c)] Compute $\Char(V-n_\lambda \Gamma_{[\lambda_1,\ldots,\lambda_m]})$.
 \item[d)] Find the leading monomial of the new polynomial. If it is not a constant, start over with b), otherwise we have the decomposition of $V$.
\end{enumerate}
After finitely many steps, we obtain the decomposition of $\Char(V)$.

\section{Dimension of $\Val^{\spin(9)}$}

In Section \ref{sec_spin9} we sketched an algorithm to determine the characters of exterior powers of some given representation of $\so(2m+1)$. In this section, we will apply this algorithm in order to compute the constants $b_k, b_{k,l}$. We denote by $V$ the $16$-dimensional space $\O^2$, with the spin representation of $\spin(9)$.

\begin{Proposition} 
The numbers $b_k=\dim (\Lambda^k(\O^2))^{\spin(9)}$ are given by 
\begin{displaymath}
 b_k =
\begin{cases}
 1 & for\ k=0,8,16,\\        
 0 &  for\ all\ other\ values\ of\ k.
\end{cases}
\end{displaymath}
\end{Proposition}

\begin{proof}  We compute the exterior powers of the spin representation of $\so(9)$ as follows. 

\emph{Case $k=0$ :} $\Lambda^0(V)=\C$ is the trivial representation, so $b_0=1$.

\emph{Case $k=1$ :} $\Lambda^1(V)=V=\Gamma_{[1/2,1/2,1/2,1/2]}$ is irreducible, so $b_1=0$. 

\emph{Case $2\leq k\leq 8$ :} With Adams formula and the algorithm described above, we can decompose $\Lambda^kV$ in irreducible representations as follows :
\begin{align*}
\Lambda^2V & =  \Gamma_{[1,1,1,0]}\oplus\Gamma_{[1,1,0,0]}\\
\Lambda^3V & = \Gamma_{[3/2,3/2,1/2,1/2]}\oplus\Gamma_{[3/2,1/2,1/2,1/2]}\\
\Lambda^4V & = \Gamma_{[2,2,0,0]}\oplus\Gamma_{[2,1,1,1]}\oplus\Gamma_{[2,1,0,0]} \oplus\Gamma_{[2,0,0,0]}\oplus\Gamma_{[1,1,1,1]}\\
\Lambda^5V & = \Gamma_{[5/2,3/2,1/2,1/2]}\oplus\Gamma_{[5/2,1/2,1/2,1/2]}\oplus \Gamma_{[3/2,3/2,3/2,3/2]}\oplus \Gamma_{[3/2,3/2,1/2,1/2]}\oplus \Gamma_{[3/2,1/2,1/2,1/2]}\\
\Lambda^6V & = \Gamma_{[3,1,1,0]}\oplus\Gamma_{[3,1,0,0]} \oplus\Gamma_{[2,2,1,1]} \oplus\Gamma_{[2,1,1,1]}\oplus\Gamma_{[2,1,1,0]} \oplus\Gamma_{[2,1,0,0]}\oplus\Gamma_{[1,1,1,0]} \oplus\Gamma_{[1,1,0,0]}\\
\Lambda^7V & = \Gamma_{[7/2,1/2,1/2,1/2]}\oplus\Gamma_{[5/2,3/2,3/2,1/2]}\oplus \Gamma_{[5/2,3/2,1/2,1/2]} \oplus\Gamma_{[5/2,1/2,1/2,1/2]}\\
 & \oplus \Gamma_{[3/2,3/2,3/2,1/2]}
 \oplus \Gamma_{[3/2,3/2,1/2,1/2]} \oplus \Gamma_{[3/2,1/2,1/2,1/2]}\oplus \Gamma_{[1/2,1/2,1/2,1/2]}\\
\Lambda^8V & = \Gamma_{[4,0,0,0]} \oplus\Gamma_{[3,1,1,1]}\oplus\Gamma_{[3,1,1,0]} \oplus\Gamma_{[3,0,0,0]} \oplus\Gamma_{[2,2,2,0]} \oplus\Gamma_{[2,2,1,0]} \oplus\Gamma_{[2,2,0,0]}\oplus\Gamma_{[2,1,1,1]}\\
& \oplus\Gamma_{[2,1,1,0]}\oplus\Gamma_{[2,0,0,0]} \oplus\Gamma_{[1,1,1,1]}\oplus\Gamma_{[1,1,1,0]} \oplus\Gamma_{[1,0,0,0]}\oplus\underbrace{\Gamma_{[0,0,0,0].}}_{\mathrm{trivial\ representation}}
\end{align*}
Hence $b_k=1$ if and only if $k=8$. 

\emph{Case $9\leq k\leq 16$ :} Since there is an isomorphism of $\spin(9)$-modules
\begin{displaymath}
 \Lambda^kV\cong \Lambda^{16-k}V,
\end{displaymath}
the only non-zero $b_k$ is $b_{16}$.
\end{proof}

\begin{Proposition} \label{prop_table_bkl}
The numbers $b_{k,l}=\dim \Omega^{k,l}_h(S\O^2)^{\spin(9)}$ are given by the following table, together with the symmetry relations $b_{k,l}=b_{15-k,15-l}=b_{k,15-l}=b_{15-k,l}$:

\begin{displaymath}
\begin{array}{|c|c||c|c|c|c|c|c|c|c|c|}
\hline  & k & 0 & 1  & 2  & 3  & 4   & 5   & 6   & 7   \\
\hline l &   &   &    &    &    &     &     &     &     \\
\hline \hline 0 &   & 1 & 0  & 0  & 1  & 2   & 1   & 0   & 4   \\
\hline 1 &   & 0 & 2  & 2  & 3  & 5   & 7   & 10  & 9   \\
\hline 2 &   & 0 & 2  & 7  & 7  & 10  & 22  & 28  & 24  \\
\hline 3 &   & 1 & 3  & 7  & 18 & 30  & 39  & 50  & 63  \\
\hline 4 &   & 2 & 5  & 10 & 30 & 56  & 68  & 88  & 116 \\
\hline 5 &   & 1 & 7  & 22 & 39 & 68  & 116 & 150 & 162 \\
\hline 6 &   & 0 & 10 & 28 & 50 & 88  & 150 & 204 & 210 \\
\hline 7 &   & 4 & 9  & 24 & 63 & 116 & 162 & 210 & 266 \\
\hline
\end{array}
\end{displaymath}
\end{Proposition}

\proof
We first describe the spaces $\Omega^{k,l}_h(SV)^{\spin(9)}$. 

Let $v=(1,0)^T\in\ \S^{15} \subset \O^2$. Since $\overline{\spin(9)}$ acts transitively on $SV$, there is an isomorphism

\begin{align*}
\Psi\ :\ \Omega^{k,l}(SV)^{\spin(9)} & \to \Lambda^{k,l}(T_{(0,v)}SV)^{\stab((0,v))}=\Lambda^{k,l}(T_{(0,v)}SV)^{\spin(7)}\\
\omega & \mapsto  \omega(0,v).
\end{align*}

The tangent space $T_{(0,v)}(SV)$ decomposes as 
\begin{displaymath}
T_{(0,v)}(SV) =\R v\oplus T\oplus T,
\end{displaymath}
where $T:=T_v\S^{15}=\O'\oplus\O$. 

The isomorphism thus becomes 
\begin{displaymath}
\Psi\ :\ \Omega^{k,l}(SV)^{\spin(9)} \to \Lambda^{k,l}(\R\oplus\O'\oplus\O\oplus\O'\oplus\O)^{\spin(7)}.
\end{displaymath}

Recall that 
\begin{displaymath}
\Omega^{k,l}_h(SV)^G := \left.\Omega^{k,l}(SV)^G \middle/\Omega^{k,l}_v(SV)^G\right., 
\end{displaymath}
where
\begin{displaymath}
 \Omega^{k,l}_v(SV)^G := \{\omega\in\Omega^{k,l}(SV)^G\, |\, \alpha\wedge \omega=0\}.
\end{displaymath}

Since 
\begin{displaymath}
\Psi(\alpha) (\lambda v+w) =  \alpha_{(0,v)}(\lambda v +w) = \langle v, \lambda v\rangle= \lambda,
\end{displaymath}
for all $w \in T \oplus T$, $\Psi$ induces an isomorphism on the space of horizontal forms by
\begin{eqnarray*}
\Omega^{k,l}_h(SV)^{\spin(9)}&\longrightarrow& \Lambda^{k,l}(\O'\oplus\O\oplus\O'\oplus\O)^{\spin(7)}\\
\omega &\mapsto & \omega(0,v).
\end{eqnarray*}

So the coefficients $b_{k,l}$ can be computed as
\begin{align*}
 b_{k,l} & =\dim\Lambda^{k,l}(\O'\oplus\O\oplus\O'\oplus\O)^{\spin(7)}\\
& = \dim (\Lambda^k(\O'\oplus\O)\otimes\Lambda^l(\O'\oplus\O) )^{\spin(7)}\\
& = \dim \Hom_{\spin(7)}(\Lambda^k(\O' \oplus \O),\Lambda^l(\O' \oplus \O)),
\end{align*}
where the last equation follows from the self-duality of $\Lambda^k(\O' \oplus \O)$.

Therefore, if  
\begin{displaymath}
 \Lambda^i(\O'\oplus \O)=\sum_\lambda n^{(i)}_\lambda \Gamma_\lambda, \quad i=0,\ldots,7
\end{displaymath}
is the decomposition into irreducible parts, Schur's lemma implies that  
\begin{equation}\label{eq2}
b_{k,l}=\sum_{\lambda}n^{(k)}_\lambda n^{(l)}_\lambda.
\end{equation} 

We first compute with Weyl's character formula \eqref{eq_weyl} the character
\begin{displaymath}
 \Char(\Gamma_{[1,0,0]})+\Char(\Gamma_{[1/2,1/2,1/2]})
\end{displaymath}
of $\O\oplus\O'$, then apply Adam's formula and the same algorithm as before. The result is the following table, whose $i$-th column contains in the line indexed by $[\lambda_1,\lambda_2,\lambda_3]$ the coefficient of $\Gamma_{[\lambda_1,\lambda_2,\lambda_3]}$ in the decomposition of $\Lambda^i(\O\oplus \O')$. 

\begin{displaymath}
 \begin{array}{| c | c | c | c | c | c | c | c | c |} \hline
 &  i=0 & i=1 & i=2 & i=3 & i=4 & i=5 & i=6 & i=7  \\ \hline
 [0,0,0] & 1 &   &   & 1 & 2 & 1 &   & 4 \\ \hline
 [\frac{1}{2},\frac{1}{2},\frac{1}{2}]  &   & 1 & 1 & 2 & 3 & 4 & 5 & 6 \\ \hline
[1,0,0] &   & 1 & 1 & 1 & 2 & 3 & 5 & 3 \\ \hline
[1,1,0] &   &   & 2 & 1 & 1 & 5 & 6 & 4 \\ \hline
 [1,1,1] &   &   &   & 2 & 4 & 3 & 4 & 7 \\ \hline
 [\frac{3}{2},\frac{1}{2},\frac{1}{2}] &   &   & 1 & 2 & 3 & 5 & 6 & 7 \\ \hline
 [\frac{3}{2},\frac{3}{2},\frac{1}{2}] &   &   &   &   1 & 2 & 3 & 5 & 5 \\ \hline
 [\frac{3}{2},\frac{3}{2},\frac{3}{2}]   &   &   &   &   & 1 & 2 & 2 & 3 \\ \hline
 [2,0,0]     &   &   &   & 1 & 2 & 1 & 1 & 3 \\ \hline
 [2,1,0]      &   &   &   & 1 & 1 & 2 & 3 & 4 \\ \hline
 [2,1,1]       &   &   &   &   & 1 & 3 & 4 & 4 \\ \hline
 [2,2,0] &   &   &   &   & 1 &   &   & 2 \\ \hline
 [2,2,1] &   &   &   &   &   & 1 & 2 & 1 \\ \hline
 [2,2,2] &   &   &   &   &   &   &   & 1 \\ \hline
 [\frac{5}{2},\frac{1}{2},\frac{1}{2}] &   &   &   &   & 1 & 1 & 2 & 2 \\ \hline
 [\frac{5}{2},\frac{3}{2},\frac{1}{2}] &   &   &   &   &   &  1 & 1 & 2 \\ \hline
 [\frac{5}{2},\frac{3}{2},\frac{3}{2}]  &   &   &   &   &   &   & 1 & 1 \\ \hline
 [3,0,0]        &   &   &   &   &   & 1 &   &   \\ \hline
 [3,1,0]        &   &   &   &   &   &   & 1 &   \\ \hline
 [3,1,1]        &   &   &   &   &   &   &   & 1 \\ \hline
 \end{array}
\end{displaymath}
 
The values for $b_{k,l}$ stated in the proposition follow from \eqref{eq2} and this table.
\endproof

\proof[Proof of Theorem \ref{mainthm_dimensions}]
Theorem \ref{mainthm_dimensions} follows from Proposition \ref{prop_table_bkl} and \eqref{eq1}. 
\endproof

\section{A canonical valuation on Riemannian manifolds}
\label{sec_valman}

We first recall the background material for valuations on manifolds. We refer to the articles \cite{alesker_val_man1}, \cite{alesker_val_man2} and the lectures notes \cite{alesker_barcelona} as a general reference for the material presented in the following.

Let $M$ be a smooth oriented $n$-dimensional manifold.
\begin{Definition}
A closed subset $P \subset M$ is a {\bf submanifold with corners} if it is locally diffeomorphic to 
$\R_{\geq 0}^i\times \R^{j}$, with integers $i,j$.
\end{Definition}

Denote by $\mathcal P (M)$ the space of compact submanifolds with corners in $M$.

Consider the oriented projectivization $\mathbb P_M$ of the cotangent bundle $T^*M$ 
\begin{displaymath}
 \mathbb P_M:= (T^*M\backslash \underbar0)/\R_{>0}\cong S^*M,
\end{displaymath}
here $\underbar0$ is the zero section of $T^*M$ and $S^*M$ denotes the cosphere bundle of $M$.

An element of $\mathbb P_M$ can be thought of as a pair $(p,H)$ with $p\in M$ and $H\subset T_pM$ an oriented hyperplane.

\begin{Definition} For $P \in\mathcal P(M)$, we define the following sets:
The {\bf tangent cone} to $P$ at $x$ is given by
\begin{displaymath}
 T_xP:=\{v \in T_xM\ |\ \exists\ \text{smooth curve}\ c:\R \to P\ \text{with}\ c(0)=x, c'(0)=v\}.
\end{displaymath}
The {\bf dual tangent cone} is
\begin{displaymath}
 T_x^\circ P:=\{\xi \in T_x^*M\ |\ \langle \xi,v \rangle \leq 0 \ \forall v \in T_xP\}.
\end{displaymath}
The {\bf normal cycle} of $P$ is given by 
\begin{displaymath}
N(P):=\bigcup_{x\in P}(T_x^\circ P\setminus \{0\})/\R_{>0}. 
\end{displaymath}
\end{Definition}

It is well known that $N(P)$ is an $(n-1)$-dimensional Lipschitz submanifold of $\mathbb P_M$, which can be oriented in a canonical way. 

If $M=V$ is a vector space, this definition coincides with Definition \ref{def:nc}.

\begin{Definition}
A {\bf valuation on a manifold} $M$ is a finitely additive functional $\mu:\mathcal P(M) \to \R$, i.e. for any $P_1,P_2 \in \mathcal P(M)$ such that $P_1 \cup P_2,\ P_1 \cap P_2 \in \mathcal{P}(M)$,
\begin{displaymath}
 \mu(P_1 \cup P_2)=\mu(P_1)+\mu(P_2)-\mu(P_1 \cap P_2).
\end{displaymath}
A valuation is said to be {\bf smooth} if there exist $\gamma \in \Omega^{n-1}(S^*M)$ and $\varphi \in \Omega^n(M)$ such that
\begin{displaymath}
\mu(P)=\int_{N(P)}\gamma+\int_P\varphi. 
\end{displaymath}
\end{Definition}

We denote by $\mathcal V^\infty(M)$ the space of smooth valuations on $M$. Note that any pair $(\varphi,\gamma) \in \Omega^{n}(M) \times \Omega^{n-1}(SM)$ defines a valuation. However, different pairs may define the same valuation, compare \cite{bernig_broecker07} for the description of the kernel of this map. 

If $\mu \in \mathcal V^\infty(M)$ is represented by a pair $(\varphi,\gamma)$, then the pair $((-1)^n \phi,(-1)^n s^*\gamma)$ defines a valuation $\sigma \mu \in \mathcal V^\infty(M)$. Here $s:S^*M \to S^*M$ is the involution $(p,[\xi]) \mapsto (p,[-\xi]), p \in M, \xi \in T_p^*M \setminus \{0\}$. The operation $\sigma$ is well-defined and is called {\bf Euler-Verdier involution} \cite{alesker_val_man2}.  

From now on, we let $(M,g)$ be a Riemannian manifold. There is a canonical identification $S^*M \cong SM$, so that we may consider $N(P)$ as a submanifold of $SM$.

It will be more convenient to work with a subset in $TM$ instead of $SM$. 

\begin{Definition}
For a compact submanifold with corners $P \subset M$,  its {\bf disc bundle} $N_1(P) \subset TM$ is the sum of 
$P \times \{0\}$ and the image of $[0,1] \times N(P)$ under the homothety in the second factor 
\begin{displaymath}
 N_1(P)=\iota_*(P)+F_*([0,1]\times N(P)),
\end{displaymath}
where $\iota : M\hookrightarrow TM,\ p\mapsto (p,0)$ is the natural inclusion and $F:\R\times SM\to TM,\ (t,(p,v))\mapsto (p,tv)$ is the 
homothety map.
\end{Definition}

Clearly $N_1(P)$ is an $n$-dimensional Lipschitz submanifold of $TM$ with boundary, and we have
\begin{displaymath}
 \partial N_1(P)=N(P).
\end{displaymath}

If a smooth form $\gamma \in \Omega^{n-1}(SM)$ extends to an $(n-1)$-form $\tilde \gamma \in \Omega^{n-1}(TM)$, then Stoke's theorem implies
\begin{displaymath}
 \int_{N(P)}\gamma=\int_{N_1(P)}d\tilde \gamma.
\end{displaymath}

Conversely, we have the following. 

\begin{Lemma}\label{lemme6.1.5}
Any smooth $n$-form $\omega$ on $TM$ defines a smooth valuation by
\begin{displaymath}
 \mu(P)=\int_{N_1(P)} \omega,
\end{displaymath}
for $P \in \P(M)$.
\end{Lemma}

\proof
Let us write  
\begin{align*}
\mu(P) & = \int_{N_1(P)} \omega\\
& =  \int_P \iota^*\omega + \int_{[0,1]\times N(P)} F^*\omega\\
& =  \int_P \underbrace{\iota^*\omega}_{=:\varphi} + \int_{N(P)}\underbrace{\int_0^1 F^*\big|_{(t,\cdot)}\left(\frac{\partial}{\partial t},\cdot\right) dt}_{=:\gamma}\\
& =  \int_P \varphi + \int_{N(P)} \gamma,
\end{align*}
with $\varphi \in \Omega^n(M)$ and $\gamma \in \Omega^{n-1}(SM)$.
\endproof

Let us denote by $\Val^\infty(TM)$ the bundle whose fiber over a point $p$ is the space $\Val^\infty(T_pM)$. Then \eqref{eq_mcmullen} implies a grading
\begin{displaymath}
\Val^\infty(TM)=\bigoplus_{k=0}^n \Val_k^\infty(TM). 
\end{displaymath}

\begin{Theorem}[\cite{alesker_val_man2, alesker_barcelona}]\label{theo:filtr}
There exists a canonical filtration of $\V^\infty(M)$ by closed subspaces
$$\V^\infty(M)=\W_0\supset \W_1\supset \cdots \supset \W_n,$$
such that the associated graded space 
$gr_W\V^\infty(M):=\bigoplus_{k=0}^n \W_k/\W_{k+1}$ is canonically isomorphic to the space $C^\infty(M, \Val^\infty (TM))$ of smooth 
sections of the infinite-dimensional vector bundle $\Val^{\infty}(TM) \longrightarrow M$.
\end{Theorem}

Let us remark that there exist product structures on $\V^\infty(M)$ and on $\Val^\infty(T_pM)$ and the isomorphism of the theorem is an isomorphism of graded algebras. However, we will not need the product structure in this paper. 

Let us describe the isomorphism more explicitly. Let $\mu \in \W_k$ and $p \in M$. Let $\tau:U \to V \subset \R^n$ be a coordinate chart around $p$. The differential 
$d\tau_p:T_pM \to T_{\tau(p)} V \cong \R^n$ is a linear isomorphism. The valuation $T_p^k \mu \in \Val_k^\infty(T_pM)$ is defined by 
\begin{equation} \label{eq_isom_expl}
 T_p^k\mu(P):=\left.\frac{1}{k!}\frac{d^k}{dt^k}\right|_{t=0}(\tau^{-1})^*\mu\left(\tau(p)+t(d\tau_p(P)-\tau(p))\right), \quad P \in \mathcal P(T_pM).
\end{equation}
It is independent of the choice of $\tau$. Strictly speaking, by this definition we obtain an element of $\mathcal{V}^\infty(T_pM)^{tr}$ (translation invariant smooth valuations on $T_pM$), but this latter space can be canonically identified with $\Val^\infty(T_pM)$ \cite{alesker_val_man2}.

Under the action of the Euler-Verdier involution, each $\W_k$ splits as $\W_k^+ \oplus \W_k^-$, where $\W_k^\pm$ denotes the $(\pm 1)$-eigenspace of $\sigma$ acting on $\W_k$. If $\mu \in \W_k^\epsilon, \epsilon \in \{\pm 1\}$, then $T_p^k \mu \in \Val_k(T_pM)$ has parity $(-1)^k \epsilon$. 

Let us introduce a filtration on the space of $n$-forms on $T^*M$, 
following \cite{alesker_val_man1}.

For every $(p,\xi)  \in T^*M$ we define
\begin{multline*}
W_k(\Omega^n(T^*M))|_{(p,\xi)} := \left\{\omega \in \Lambda^n T_{(p,\xi)}^*(T^*M)\ |\ \omega|_F \equiv 0\ \text{ for all }\ F\subset T_{(p,\xi)}(T^*M)\right.\\
 \left.\text{with}\ \dim(F\cap T_{(p,\xi)} (\pi^{-1}(p)))>n-k\right\},
\end{multline*}
where $\pi:T^*M\to M$ is the projection. Then we have the filtration
\begin{multline*}
\Omega^n(T^*M) = W_0(\Omega^n(T^*M)) \supset W_1(\Omega^n(T^*M)) \supset \ldots \\
\ldots \supset W_n(\Omega^n(T^*M)) \supset W_{n+1}(\Omega^n(T^*M))=0. 
\end{multline*}

For $\epsilon \in \{\pm 1\}$ set 
\begin{displaymath}
W_k^\epsilon(\Omega^n(T^*M)):=\{\omega \in W_k^\epsilon(\Omega^n(T^*M)): \tilde s^* \omega=(-1)^{n+\epsilon} \omega\}, 
\end{displaymath}
where $\tilde s:T^*M \to T^*M, (x,\xi) \mapsto (x,-\xi)$. 

\begin{Theorem}[\cite{alesker_val_man1}]\label{th:intcc}
 The map $\Xi:\Omega^n(TM) \to \V^\infty(M)$ given by
\begin{displaymath}
 (\Xi(\omega))(P):=\int_{N_1(P)}\omega
\end{displaymath}
is surjective. More precisely, for $k \in \{0,1,\ldots,n\}$ and $\epsilon \in \{\pm 1\}$, the map $\Xi$ maps $W_k^\epsilon(\Omega^n(TM))$ onto $\W_k^\epsilon$ surjectively.
\end{Theorem}

Let now $(M,g)$ be a Riemannian manifold. Let $\Gamma(M)$ denote the space of vector fields on $M$, let $R$ be the Riemann curvature tensor of $M$ and $K$ its sectional curvature.

For $p\in M$, the tensor $R_p$ is an element of $\mathrm{Sym}^2\Lambda^2 T_pM \subset \Lambda^2 T_pM \otimes \Lambda^2 T_p M$. Let $R^*_p$ be the image of $R_p$ under the maps
\begin{displaymath}
 \Lambda^2 T_pM \otimes \Lambda^2 T_p M \stackrel{\cong} {\longrightarrow} \Lambda^2 T_pM \otimes \Lambda^{n-2} T_p M \subset \Lambda^n(T_pM \oplus T_pM) \cong \Lambda^n(T_{(p,v)}TM),
\end{displaymath}
where the first isomorphism is induced by the Hodge-$*$ operator.

We define an $n$-form $\omega\in\Omega^n(TM)$ by
\begin{displaymath}
 \omega_{(p,v)}:=\frac{1}{\kappa_{n-2}} R_p^*\in\Lambda^n(T_{(p,v)}TM),
\end{displaymath}
where the normalization constant $\kappa_{n-2}$ is the volume of the $(n-2)$-dimensional unit ball, that is
\begin{displaymath}
 \kappa_{n-2}=\frac{\pi^{(n-2)/2}}{\Gamma\left(\frac{n-2}{2}+1\right)}.
\end{displaymath}

By Lemma \ref{lemme6.1.5}, we may associate a smooth valuation to $\omega$. 
\begin{Definition}
The sectional curvature valuation $\riem \in\mathcal V^\infty (M)$ is defined by 
\begin{displaymath}
\riem(P):=\int_{N_1(P)}\omega, \quad P \in \mathcal{P}(V).
\end{displaymath} 
\end{Definition}

\begin{Theorem}
The valuation $\riem$ defined above has filtration index $2$ and belongs to the $(+1)$-eigenspace of the Euler-Verdier involution, i.e.\,$\riem \in \W_2^+$. For each $p \in M$, the Klain function of the induced valuation $T_p^2 \riem \in \Val_2^+(T_pM)$ equals the sectional curvature of $M$ at $p$.  
\end{Theorem}

\proof
In order to show that $\riem \in \W_2$, we use Theorem \ref{th:intcc}. It thus suffices to show that, for $(p,v) \in TM$
\begin{displaymath}
 R_p^*\big|_F \equiv 0
\end{displaymath}
for all $F\subset T_{(p,v)} (TM)$ with $\dim(F \cap T_{(p,v)} (\pi^{-1} p))= \dim(F\cap V_{(p,v)})>n-2$, where $V_{(p,v)} \cong T_pM$ is the vertical 
subspace of $T_{(p,v)}(TM)$. However, this is immediate from the fact that $R_p^* \in \Lambda^2(T_pM) \otimes \Lambda^{n-2}(T_pM)$.

Since $\tilde s^*\omega=(-1)^n \omega$, $\riem \in \W_2^+$ and hence $T_p^2 \riem \in \Val_2^{\infty,+}(T_pM)$ for all $p \in M$.

Finally, fix $p \in M$ and let us compute the Klain function of $T_p^2 \riem \in \Val_2^+(T_pM)$.

Let $E \in \Gr_2(T_pM)$, and let $D^2$ be the $2$-dimensional unit ball in $E$. Consider the exponential 
map $\exp:T_pM \to M$, and set $\tau:=\exp^{-1}$. Then we have $\tau(p)=0$ and $d\tau_p=\mathrm{id}|_{T_pM}$. Using \eqref{eq_isom_expl} we obtain 

\begin{align*}
 T_p^2\riem(D^2) & = \left.\frac{1}{2}\frac{d^2}{dt^2}\right|_{t=0}(\tau^{-1})^*\riem \left(\tau(p)+t(d\tau_p(D^2)-\tau(p))\right)\\
 & = \left.\frac{1}{2}\frac{d^2}{dt^2}\right|_{t=0} \riem \left(\tau^{-1}(tD^2)\right)\\
 & = \left.\frac{1}{2} \frac{d^2}{dt^2}\right|_{t=0} \int_{N_1(\tau^{-1}(tD^2))}\omega.
 \end{align*}

Define $S_t:=\tau^{-1}(tD^2)\in \mathcal P(M)$. Then
\begin{align*}
\riem(S_t) & = \int_{N_1(S_t)}\omega\\
& = \int_{N_1(S_t)} \omega_{(q,v)}(v_1^h,v_2^h,v_3^v,\ldots,v_n^v) dqdv,
\end{align*}
where $v_1^h,v_2^h$ are horizontal lifts (i.e.\,lifts in the horizontal subspace $H_{(q,v)}$) of an orthonormal basis $\{v_1,v_2\}$ of $T_qS_t$ and $v_3^v,\ldots,v_n^v$ are vertical lifts (i.e.\,lifts in the vertical subspace $V_{(q,v)}$) of an orthonormal basis of the orthogonal complement of $T_qS_t$ in $T_qM$. 

The definition of $\omega$ implies that 
\begin{displaymath}
\riem(S_t) = \frac{1}{\kappa_{n-2}}\vol_{n-2}(D^{n-2})\int_{S_t} R_q^*(v_1,v_2,v_3,\ldots,v_n) dq,
\end{displaymath} 
where $D^{n-2}$ denotes the $(n-2)$-dimensional unit ball, since, by definition, $R_q^*$ is constant on each fiber. 

Using the isomorphism given by the Hodge-$*$ operator, we obtain
\begin{align*}
\frac{1}{\kappa_2 t^2}\riem(S_t) & = \frac{1}{\kappa_2 t^2}\int_{S_t} R_q(v_1,v_2,v_1,v_2) dq\\
& = \frac{1}{\kappa_2 t^2} \int_{S_t} K(T_qS_t) dq\\
& \longrightarrow  K(E),\quad \text{as}\ t\to 0.
\end{align*}

By definition of $T_p^2 \riem$, we obtain 
\begin{align*}
T_p^2\riem(D^2) & = \left.\frac{1}{2}\frac{d^2}{dt^2}\right|_{t=0}\riem\left(S_t\right)\\
& =  \kappa_2 K(E)\\
& = \vol_2(D^2) K(E),
\end{align*}
and therefore 
\begin{displaymath}
 \Kl_{T_p^2 \riem}(E)=K(E),
\end{displaymath}
i.e. the Klain function of $T_p^2\riem$ coincides with the sectional curvature $K$ of $M$ at $p$.
\endproof

\subsection*{Examples} 

Let $G$ be a Lie group acting isotropically on $M$, i.e. $G$ acts transitively on the unit sphere bundle $SM$. Let us fix a point $p \in M$ and let $H \subset G$ be the stabilizer of $p$. Then the isomorphism from Theorem \ref{theo:filtr} restricts to an isomorphism 
\begin{displaymath}
 gr_W\V^\infty(M)^G\cong C^\infty(M, \Val^\infty (TM))^G\cong \Val^H(T_pM).
\end{displaymath}
In particular, we have the following isomorphism
\begin{equation}\label{eq:filtration}
(\W_k/\W_{k+1})^G \cong \Val^H_k(T_pM), \rho \mapsto T_p^k \rho.
\end{equation}

Since the form $\omega$ is $G$-invariant, the same holds true for the valuation $\riem$. Hence $T_p^2\riem \in \Val^H_2(T_pM)$.

\begin{enumerate}
 \item \textbf{The complex projective space.}  

Let $M:=\C P^n$ with its Fubini-Study metric, $T_pM=\C^n$, $G:=\mathrm U(n+1)$, $H:=\mathrm{Stab}_p=\mathrm{U}(1) \times \mathrm U(n)$. The action of $H$ on $\mathbb{C}^n$ is not faithful, factoring out the kernel leaves us with the canonical action of $\mathrm{U}(n)$ on $\mathbb{C}^n$.

 It is well-known  that the $\mathrm{U}(n)$-orbits of $\Gr_2(\C^n)$ are characterized by their K\"ahler angles (compare e.g. Tasaki \cite{tasaki00} for a more general statement). More precisely, let a plane $E$ be generated by two orthogonal vectors $v,w$. Then the K\"ahler angle $\varphi(E) \in \left[0,\frac{\pi}{2}\right]$ is defined by the equation 
\begin{displaymath}
\cos^2\varphi(E)=\langle v,iw \rangle^2. 
\end{displaymath}
Thus $\varphi$ measures the angle between the complex planes spanned by $\{v,iv\}$ and $\{w,iw\}$ (cf. \cite{klingenberg_book}).

The sectional curvature of a plane $E$ is given by $K(E)=1+3\cos^2 \varphi(E)$, hence the Klain function of $T_p^2 \riem$ is given by
\begin{displaymath}
\Kl_{T_p^2\riem} (E) = 1+3\cos^2\varphi(E), 
\end{displaymath}

Let us relate the sectional curvature valuation 
\begin{displaymath}
T_p^2\riem \in \Val_2(T_p \mathbb{CP}^n)^{\mathrm{U}(n)}=\Val_2(\mathbb{C}^n)^{\mathrm{U}(n)} 
\end{displaymath}
with known valuations from hermitian integral geometry. Several bases of $\Val^{\mathrm U(n)}$ have been constructed in \cite{bernig_fu_hig}, for instance {\bf hermitian intrinsic volumes} and {\bf Tasaki valuations}. In particular, a basis for $\Val_2^{\mathrm U(n)}$ is given by 
the Tasaki valuations $\tau_{2,i},\,i=0,1$. Their Klain functions are
\begin{displaymath}
\Kl_{\tau_{2,0}}=1,\qquad\qquad \Kl_{\tau_{2,1}}=\cos^2\varphi, 
\end{displaymath}
where $\varphi$ is the K\"ahler angle. Hence the sectional curvature valuation $T_p^2\riem$ can be written in terms of the Tasaki valuations as
\begin{displaymath}
 T_p^2\riem=\tau_{2,0}+3\tau_{2,1}.
\end{displaymath}
\item Let $M:=\mathbb{HP}^n$, the quaternionic projective space with its standard metric, $G:=\mathrm{Sp}(n+1)$, $T_pM \cong \mathbb{H}^n$, $H:=\mathrm{Stab}_p=\mathrm{Sp}(n) \times \mathrm{Sp}(1)$. Again the action of $H$ on $T_pM$ is not faithful. Factoring out the kernel we obtain an action of $\mathrm{Sp}(n) \cdot \mathrm{Sp}(1)$ on $\mathbb{H}^n$.

The $\mathrm{Sp}(n) \cdot \mathrm{Sp}(1)$-orbits on $\Gr_2(\mathbb{H}^n)$ were described in \cite{bernig_solanes}. Given a $2$-plane $E \subset \mathbb{H}^n$, choose an orthonormal basis $u_1,u_2$ of $E$. Then the standard quaternionic hermitian scalar product of $u_1,u_2$ is a pure quaternion. The orbit of $E$ is uniquely characterized by the norm $\lambda$ of this quaternion. Moreover, there exist an $\mathrm{Sp}(n) \cdot \mathrm{Sp}(1)$ and translation invariant continuous valuation $\tau$ whose Klain function equals $\lambda^2$. These statements are shown in the case $n=2$ in \cite[Theorems 1 and 2]{bernig_solanes}, but the proofs work for higher $n$ as well. 

On the other hand, the sectional curvature of $\mathbb{HP}^n$ at the $2$-plane $E$ equals $1+3\lambda^2$, compare e.g. \cite{kraines66}. Since $\dim \Val_2(\mathbb{H}^n)^{\mathrm{Sp}(n) \cdot \mathrm{Sp}(1)}=2$ \cite{bernig_quat09}, we obtain by comparing the Klain functions 
\begin{displaymath}
 T_p^2\riem=\mu_2+3\tau,
\end{displaymath}
where $\mu_2$ is the second intrinsic volume (whose Klain function is $1$).

\item \textbf{The octonionic projective plane.}
Due to the non-associativity of $\O$, the concept of octonionic projective 
space only makes sense in dimension $2$ (\cite{baez}). For $M=\O P^2$, we have $T_pM=\O^2$, $G:=\mathrm{F}_4$ and $H:=\mathrm{Stab}_p=\spin(9)$. 

For $(a,b),(c,d)\in\O^2$ with $\|(a,b)\|=\|(c,d)\|=1$ and $\langle (a,b),(c,d)\rangle=0$, the sectional curvature of the plane generated by 
$(a,b), (c,d)$ is given by (cf.\,\cite{brown_gray})
\begin{align*}
K(E_{(a,b),(c,d)}) & = 4\left[\| a\wedge c\|^2+\| b\wedge d\|^2+\frac{1}{4}\|a\|^2\|d\|^2+\frac{1}{4}\|b\|^2\|c\|^2\right.\\
 & \left.\ \ \ +\frac{1}{2}\langle ab,cd\rangle-\langle ad,bc\rangle\right],
\end{align*}
where $\| a\wedge b\|^2$ is
\begin{displaymath}
 \| a\wedge b\|^2=\det\left(
\begin{array}{cc}
\langle a,a\rangle & \langle a,b \rangle\\
\langle b,a\rangle &  \langle b,b \rangle
\end{array}\right)=\|a\|^2\|b\|^2-\langle a,b\rangle^2.
\end{displaymath}

In particular,
\begin{displaymath}
K(E_{(1,0),(\mathbf{i},0)})=4, \quad K(E_{(1,0),(0,1)})=1.
\end{displaymath}

Alesker \cite{alesker08} constructed a valuation $\tau_{\text{oct}}$ on $\O^2$ which is $\spin(9)$-invariant and of degree of homogeneity $2$, called the {\bf octonionic pseudo-volume}. Its Klain function satisfies 
\begin{displaymath} 
\Kl_{\tau_{\text{oct}}}(E_{(1,0),(\mathbf{i},0)})=0, \quad \Kl_{\tau_{\text{oct}}}(E_{(1,0),(0,1)})=1. 
\end{displaymath}

Since we have shown that $\Val_2^{\spin(9)}$ is of dimension $2$, the valuation $T_p^2\riem$ can be expressed as a linear combination of the second intrinsic volume and the octonionic pseudo-volume. Comparing the values of the Klain functions, we obtain 
\begin{displaymath}
T_p^2 \riem=4\mu_2-3\tau_{\text{oct}}. 
\end{displaymath}

\end{enumerate}

\def\cprime{$'$}


\begin{thebibliography}{10}

\bibitem{alesker00}
Semyon Alesker.
\newblock On {P}. {M}c{M}ullen's conjecture on translation invariant
  valuations.
\newblock {\em Adv. Math.}, 155(2):239--263, 2000.

\bibitem{alesker_mcullenconj01}
Semyon Alesker.
\newblock {Description of translation invariant valuations on convex sets with
  solution of P. McMullen's conjecture.}
\newblock {\em Geom. Funct. Anal.}, 11(2):244--272, 2001.

\bibitem{alesker04_product}
Semyon Alesker.
\newblock The multiplicative structure on continuous polynomial valuations.
\newblock {\em Geom. Funct. Anal.}, 14(1):1--26, 2004.

\bibitem{alesker_su2_04}
Semyon Alesker.
\newblock {SU(2)-invariant valuations.}
\newblock {Milman, V. D. (ed.) et al., Geometric aspects of functional
  analysis. Papers from the Israel seminar (GAFA) 2002--2003. Berlin: Springer.
  Lecture Notes in Mathematics 1850, 21-29 (2004)}, 2004.

\bibitem{alesker_val_man1}
Semyon Alesker.
\newblock {Theory of valuations on manifolds. I: Linear spaces.}
\newblock {\em Isr. J. Math.}, 156:311--339, 2006.

\bibitem{alesker_val_man2}
Semyon Alesker.
\newblock Theory of valuations on manifolds. {II}.
\newblock {\em Adv. Math.}, 207(1):420--454, 2006.

\bibitem{alesker_survey07}
Semyon Alesker.
\newblock Theory of valuations on manifolds: a survey.
\newblock {\em Geom. Funct. Anal.}, 17(4):1321--1341, 2007.

\bibitem{alesker_val_man4}
Semyon Alesker.
\newblock Theory of valuations on manifolds. {IV}. {N}ew properties of the
  multiplicative structure.
\newblock In {\em Geometric aspects of functional analysis}, volume 1910 of
  {\em Lecture Notes in Math.}, pages 1--44. Springer, Berlin, 2007.

\bibitem{alesker08}
Semyon Alesker.
\newblock Plurisubharmonic functions on the octonionic plane and
  {S}pin(9)-invariant valuations on convex sets.
\newblock {\em J. Geom. Anal.}, 18(3):651--686, 2008.

\bibitem{alesker_fourier}
Semyon Alesker.
\newblock A {F}ourier type transform on translation invariant valuations on
  convex sets.
\newblock {\em Israel J. Math.}, 181:189--294, 2011.

\bibitem{alesker_barcelona}
Semyon Alesker.
\newblock New structures on valuations and applications.
\newblock In Eduardo Gallego and Gil Solanes, editors, {\em Integral Geometry
  and Valuations}, Advanced Courses in Mathematics - CRM Barcelona, pages
  1--45. Springer Basel, 2014.

\bibitem{alesker_val_man3}
Semyon Alesker and Joseph H.~G. Fu.
\newblock Theory of valuations on manifolds. {III}. {M}ultiplicative structure
  in the general case.
\newblock {\em Trans. Amer. Math. Soc.}, 360(4):1951--1981, 2008.

\bibitem{baez}
John~C. Baez.
\newblock The octonions.
\newblock {\em Bull. Amer. Math. Soc. (N.S.)}, 39(2):145--205, 2002.

\bibitem{bernig_sun09}
Andreas Bernig.
\newblock {A Hadwiger type theorem for the special unitary group.}
\newblock {\em Geom. Funct. Anal.}, 19:356--372, 2009.

\bibitem{bernig_quat09}
Andreas Bernig.
\newblock {A product formula for valuations on manifolds with applications to
  the integral geometry of the quaternionic line.}
\newblock {\em Comment. Math. Helv.}, 84(1):1--19, 2009.

\bibitem{bernig_g2}
Andreas Bernig.
\newblock Integral geometry under {$G_2$} and {${\rm Spin}(7)$}.
\newblock {\em Israel J. Math.}, 184:301--316, 2011.

\bibitem{bernig_aig10}
Andreas Bernig.
\newblock Algebraic integral geometry.
\newblock In {\em Global Differential Geometry}, volume~17 of {\em Springer
  Proceedings in Mathematics}, pages 107--145. Springer, Berlin Heidelberg,
  2012.

\bibitem{bernig_qig}
Andreas Bernig.
\newblock {Invariant valuations on quaternionic vector spaces}.
\newblock {\em J. Inst. Math. Jussieu}, 11:467--499, 2012.

\bibitem{bernig_broecker07}
Andreas Bernig and Ludwig Br{\"o}cker.
\newblock {Valuations on manifolds and Rumin cohomology.}
\newblock {\em J. Differ. Geom.}, 75(3):433--457, 2007.

\bibitem{bernig_fu06}
Andreas Bernig and Joseph H.~G. Fu.
\newblock Convolution of convex valuations.
\newblock {\em Geom. Dedicata}, 123:153--169, 2006.

\bibitem{bernig_fu_hig}
Andreas Bernig and Joseph H.~G. Fu.
\newblock Hermitian integral geometry.
\newblock {\em Ann. of Math.}, 173:907--945, 2011.

\bibitem{bernig_fu_solanes}
Andreas Bernig, Joseph H.~G. Fu, and Gil Solanes.
\newblock Integral geometry of complex space forms.
\newblock {\em Geom. Funct. Anal.}, 24(2):403--492, 2014.

\bibitem{bernig_hug}
Andreas Bernig and Daniel Hug.
\newblock {Kinematic formulas for tensor valuations}.
\newblock Preprint arxiv:1402.2750.

\bibitem{bernig_solanes}
Andreas Bernig and Gil Solanes.
\newblock {Classification of invariant valuations on the quaternionic plane.}
\newblock {\em J. Funct. Anal.}, 267:2933--2961, 2014.

\bibitem{borel49}
Armand Borel.
\newblock Some remarks about {L}ie groups transitive on spheres and tori.
\newblock {\em Bull. Amer. Math. Soc.}, 55:580--587, 1949.

\bibitem{broecker_tomdieck}
Theodor Br{\"o}cker and Tammo tom Dieck.
\newblock {\em Representations of compact {L}ie groups}, volume~98 of {\em
  Graduate Texts in Mathematics}.
\newblock Springer-Verlag, New York, 1995.
\newblock Translated from the German manuscript, Corrected reprint of the 1985
  translation.

\bibitem{brown_gray}
Robert~B. Brown and Alfred Gray.
\newblock Riemannian manifolds with holonomy group {${\rm S}pin$} (9).
\newblock In {\em Differential geometry (in honor of {K}entaro {Y}ano)}, pages
  41--59. Kinokuniya, Tokyo, 1972.

\bibitem{fu94}
Joseph H.~G. Fu.
\newblock Curvature measures of subanalytic sets.
\newblock {\em Amer. J. Math.}, 116(4):819--880, 1994.

\bibitem{fu_barcelona}
Joseph~H.G. Fu.
\newblock Algebraic integral geometry.
\newblock In Eduardo Gallego and Gil Solanes, editors, {\em Integral Geometry
  and Valuations}, Advanced Courses in Mathematics - CRM Barcelona, pages
  47--112. Springer Basel, 2014.

\bibitem{fulton_harris91}
William Fulton and Joe Harris.
\newblock {\em {Representation theory. A first course.}}
\newblock {Graduate Texts in Mathematics. 129. New York etc.: Springer-Verlag,.
  xv, 551 p., 144 ill. }, 1991.

\bibitem{hadwiger_vorlesung}
Hugo Hadwiger.
\newblock {\em Vorlesungen \"uber {I}nhalt, {O}berfl\"ache und
  {I}soperimetrie}.
\newblock Springer-Verlag, Berlin-G\"ottingen-Heidelberg, 1957.

\bibitem{klain95}
Daniel~A. Klain.
\newblock A short proof of {H}adwiger's characterization theorem.
\newblock {\em Mathematika}, 42(2):329--339, 1995.

\bibitem{klain_rota}
Daniel~A. Klain and Gian-Carlo Rota.
\newblock {\em Introduction to geometric probability}.
\newblock Lezioni Lincee. [Lincei Lectures]. Cambridge University Press,
  Cambridge, 1997.

\bibitem{klingenberg_book}
Wilhelm Klingenberg.
\newblock {\em Riemannian geometry}, volume~1 of {\em de Gruyter Studies in
  Mathematics}.
\newblock Walter de Gruyter \& Co., Berlin-New York, 1982.

\bibitem{kraines66}
Vivian~Yoh Kraines.
\newblock Topology of quaternionic manifolds.
\newblock {\em Trans. Amer. Math. Soc.}, 122:357--367, 1966.

\bibitem{mcmullen77}
Peter McMullen.
\newblock Valuations and {E}uler-type relations on certain classes of convex
  polytopes.
\newblock {\em Proc. London Math. Soc. (3)}, 35(1):113--135, 1977.

\bibitem{montgomery_samelson43}
Deane Montgomery and Hans Samelson.
\newblock Transformation groups of spheres.
\newblock {\em Ann. of Math. (2)}, 44:454--470, 1943.

\bibitem{park02}
Heungii Park.
\newblock {Kinematic formulas for the real subspaces of complex space forms of
  dimension $2$ and $3$.}
\newblock PhD-thesis University of Georgia 2002.

\bibitem{rumin94}
Michel Rumin.
\newblock {Differential forms on contact manifolds. (Formes diff{\'e}rentielles
  sur les vari{\'e}t{\'e}s de contact.)}.
\newblock {\em J. Differ. Geom.}, 39(2):281--330, 1994.

\bibitem{sudbery83}
Anthony Sudbery.
\newblock Octonionic description of exceptional {L}ie superalgebras.
\newblock {\em J. Math. Phys.}, 24(8):1986--1988, 1983.

\bibitem{sudbery84}
Anthony Sudbery.
\newblock Division algebras, (pseudo)orthogonal groups and spinors.
\newblock {\em J. Phys. A}, 17(5):939--955, 1984.

\bibitem{tasaki00}
Hiroyuki Tasaki.
\newblock Generalization of {K}{\"a}hler angle and integral geometry in complex
  projective spaces.
\newblock In {\em Steps in differential geometry ({D}ebrecen, 2000)}, pages
  349--361. Inst. Math. Inform., Debrecen, 2001.

\bibitem{voide_thesis}
Floriane Voide.
\newblock {$\spin(9)$-invariant valuations.}
\newblock PhD-thesis Goethe University Frankfurt 2013.

\bibitem{wannerer_area_measures}
Thomas Wannerer.
\newblock Integral geometry of unitary area measures.
\newblock {\em Adv. Math.}, 263:1--44, 2014.

\bibitem{wannerer_unitary_module}
Thomas Wannerer.
\newblock The module of unitarily invariant area measures.
\newblock {\em J. Differential Geom.}, 96(1):141--182, 2014.

\end{thebibliography}

\end{document}